\documentclass[preprint,12pt]{elsarticle}

\usepackage{graphicx}
\usepackage[left=3cm,right=1.5cm, top=2cm,bottom=2cm,bindingoffset=0cm]{geometry}
\usepackage{amssymb}
\usepackage{amsthm}
\usepackage{cmap}

\usepackage{amsmath}

\biboptions{sort&compress}

\newcommand{\sysn}{\left\{\begin{array}{rcl}}
\newcommand{\sysk}{\end{array}\right.}

\newtheorem{theorem}{Theorem}[section]
\newtheorem{lemma}[theorem]{Lemma}

\theoremstyle{definition}
\newtheorem{example}[theorem]{Example}

\theoremstyle{plain}
\newtheorem{proposition}[theorem]{Proposition}
\theoremstyle{definition}
\newtheorem{definition}[theorem]{Definition}

\newtheorem{corollary}[theorem]{Corollary}

\journal{...}

\begin{document}

\begin{frontmatter}


\title{Joint continuity in semitopological monoids and semilattices}
\tnotetext[t1]{The research was supported by the Russian Science Foundation (RSF Grant No. 23-21-00195).}
\author{Alexander V. Osipov}

\ead[label2]{OAB@list.ru}

\author{Konstantin Kazachenko}


\ead[label1]{voice1081@gmail.com}

\address{Krasovskii Institute of Mathematics and Mechanics, Yekaterinburg, Russia}


\begin{abstract} In this paper we study the separately continuous actions of semitopological monoids
on pseudocompact spaces. The main aim of this paper is to generalize Lawson's results to some class of pseudocompact spaces.
Also, we introduce a concept of a weak $q_D$-space and
prove that a pseudocompact space and a weak $q_D$-space form a Grothendieck
pair.

As an application of the main result, we investigate the continuity of multiplication and taking inverses in subgroups of semitopological
semigroups. In particular, we get that if $(S, \bullet)$ is  a Tychonoff pseudocompact semitopological monoid with a quasicontinuous multiplication $\bullet$  and $G$ is a subgroup of $S$, then $G$ is a topological group.

Also, we study the continuity of operations in semitopological semilattices.
\end{abstract}


\end{frontmatter}



\section{Introduction}

Deriving points of continuity of a separately continuous function from topological properties of its domain and range is a classic topic in general topology. One of the first significant results in this direction belongs to H. Hahn \cite{Hahn}.

\begin{theorem}(Hahn). Let $X$ be a complete metric space, $Y$ a compact metric space and $f:X\times Y\rightarrow \mathbb R$ a separately continuous map. Then there exists a dense $G_{\delta}$-set $A$ in $X$, such that $f$ is
jointly continuous at each point of $A\times Y$.
\end{theorem}

In 1974, I. Namioka obtained a generalization of Hahn's theorem \cite{Namioka}.

\begin{theorem}(Namioka). Let $X$ be a strongly countably complete regular
space, $Y$ a locally compact and $\sigma$-compact space, $Z$ a pseudo-metrizable space and $f:X\times Y\rightarrow Z$ a separately continuous map. Then there exists a dense $G_{\delta}$-set $A$ in $X$, such that $f$ is
jointly continuous at each point of $A\times Y$.
\end{theorem}

A decent amount of studies of the separate continuity property of functions were motivated by Namioka’s result. Moreover, a new class of topological spaces was defined – \textit{Namioka's spaces}.

\medskip

An important result in this direction is the following theorem by N. Bourbaki~\cite{Bourbaki}.

\begin{theorem}(Bourbaki). Let $X$ be a Baire space, $y_0 \in Y$ has a countable base, $Z$ a metric space and $\pi : X \times Y \to Z$ a separately continuous function. Then the set of points $x \in X$ for which $\pi$ is discontinous in $(x, y_0)$ is of first Baire category.

\end{theorem}

Among later studies, the works of E.A.~Reznichenko stand out, where he explores the interplay between Namioka’s theorem and generalizations of the classical Grothendieck theorem (see more about Grothendieck’s theorem in \cite{Arh1,Os}). In particular, the following results were proven in \cite{Reznichenko}.

\begin{theorem}(Reznichenko)\label{contCriterion}
    Let $X$ and $Y$ be Tychonoff pseudocompact spaces and $\Phi : X \times Y \to \mathbb{R}$ be a separately continuous function. Then $\Phi$ extends to a separately continuous function on $\beta X \times \beta Y$ if and only if there is a dense $G_\delta$-set $E \subset X$ such that $\Phi$ is jointly continuous at each point of $E \times Y$.
\end{theorem}

\begin{theorem}(Reznichenko)\label{contCriterion1}
    Let $X$ and $Y$ be Tychonoff pseudocompact spaces, $Z$ a space and $\Phi : X \times Y \to Z$ a separately continuous function. Then $\Phi$ extends to a separately continuous function  $\Tilde{\Phi} : \beta X \times \beta Y \to \mu Z$ if and only if $\Phi$ is quasicontinuous.
\end{theorem}

Here we denote the Stone–Čech compactification of $X$ by $\beta X$ and the Hewitt realcompactification of $Z$ by $\mu Z$.

The above theorems have found applications in topological algebra. One of the first to use them in his research was J.D.~Lawson  \cite{LawsonSep}, who obtained very interesting results on the joint continuity of the multiplication in topologized semigroups. In particular, Lawson attained the following results.

\begin{theorem}(Lawson) {\it Let $S$ be a compact Hausdorff right semitopological semigroup with an identity ${\bf e}$, $X$ a compact Hausdorff space, and  $\pi: S\times X\rightarrow X$ a separately continuous action. If $g$ is a unit in $S$, then $\pi$ is continuous at $(g,x)$ for all $x\in X$.}
\end{theorem}

\begin{corollary}(Lawson) {\it Let $G$ be a subgroup of a compact Hausdorff semitopological
semigroup. Then $G$ is a topological group.}
\end{corollary}

The main aim of this paper is to generalize Lawson's  results to some class of pseudocompact spaces and, as an application of the main result, we investigate the continuity of multiplication and taking inverses in subgroups of semitopological monoids and semilattices.

\section{Preliminaries}

Let us recall some definitions. A function $f : X \times Y \to Z$ is called \textit{separately continuous}, if functions $f_x : Y \to Z, f_x(y) = f(x, y)$ and $f_y : X \to Z, f_y(x) = f(x, y)$ are continuous for each $x \in X$ and $y \in Y$.

 A semigroup endowed with topology a is called \textit{topologized} semigroup. A topologized semigroup $X$ is called \textit{right topological semigroup}, if map $s_x : X \to X$, defined as $s_x(y) = xy$ is continuous for each $x \in X$.

 \textit{Semitopological} semigroup is a topologized semigroup with a separately continuous multiplication. The map $\pi : S \times X \to X$, where $S$ is a semigroup and $X$ is a set is called an \textit{action} of $S$ on $X$, if for each $s_1, s_2 \in S$ and $x \in X$ we have $\pi(s_1s_2, x) = \pi(s_1, \pi(s_2, x))$.

\textit{Semitopological monoid} is a topologized semigroup with an identity and a separately continuous multiplication.

\medskip

Let $X$ and $Y$ be topological spaces. A function $f: X \rightarrow Y$ is {\it quasicontinuous}  at $x\in X$ if for every
open set $V\subset Y$ where $f(x)\in V$ and for every open set $U\subset X$ where $x\in U$ there is a nonempty open set $W\subset U$ such that
$f(W)\subset V$ \cite{Neub}. If $f$ is quasicontinuous at every point of $X$, we say that $f$ is quasicontinuous.
We say that a subset of $X$ is {\it quasi-open (or semi-open)}  if it is contained in the closure of its interior.
Then a function $f : X \rightarrow Y$ is quasicontinuous if and only if $f^{-1}(V)$ is quasi-open for every open set $V\subset Y$ \cite{Neub}.

All spaces are assumed to be Tychonoff unless otherwise stated. For every space $X$ we denote the set of the real-valued continuous functions on $X$ and the space of continuous functions on $X$ in the topology of pointwise convergence by $C(X)$ and $C_p(X)$ respectively.

Recall that a topological space $X$ is {\it pseudocompact}, if every real-valued continuous function on $X$ is bounded.

\section{Main results}

In this section we obtain a generalisation of Lawson's central theorem in \cite{LawsonSep} and some of its corollaries.

\begin{theorem}\label{1contTheorem}
    Let $S$ be a pseudocompact right topological monoid, $X$ a pseudocompact space, $\pi : S \times X \to X$ a separately continuous quasicontinuous action such that $\pi(1, x) = x$ for each $x \in X$. Then $\pi$ is continuous at each point of $\lbrace 1 \rbrace \times X$.
\end{theorem}

\begin{proof} By Theorem \ref{contCriterion1}, $\pi$ extends to a separately continuous function $\Tilde{\pi} : \beta S \times \beta X \to \mu X$. Since $X$ is pseudocompact, $\mu X = \beta X$ (see ex. 3.11.C in \cite{Engelking}).

    Suppose $\pi$ is discontinuous at some point $(1, x)$. Then there are nets $u_\alpha \to 1, x_\alpha \to x$, such that $\pi(u_\alpha, x_\alpha)$ does not converge to $\pi(1, x) = x$. It means that there is an open neighbourhood $U$ of $x$ and a subnet $\lbrace \pi(u_\beta, x_\beta) \rbrace_{\beta \in \Lambda}$ such that $\pi(u_\beta, x_\beta) \not\in U$ for each $\beta \in \Lambda$. Since $\beta X$ is compact,  $\lbrace \pi(u_\beta, x_\beta) \rbrace_{\beta \in \Lambda}$ has a converging subnet $\lbrace \pi(u_\gamma, x_\gamma) \rbrace_{\gamma \in \Gamma}$. Let $\pi(u_\gamma, x_\gamma)$ converge to $z$. Since $z \neq x$, we can, making $U$ smaller if needed, assume that $z \not\in U$. Since $\beta X$ is Tychonoff, there is a continuous $\Tilde{f} : \beta X \to I$ such that $\Tilde{f}(x) = 1, \Tilde{f}(X \setminus U) = \lbrace 0 \rbrace$. We set $W = U \cap \Tilde{f}^{-1}((\frac{1}{2}, 1])$, $V = \Tilde{f}^{-1}([0, \frac{1}{2}))$ and $\Tilde{f} \restriction_{X} = f$. Note that $z \in V$ and $\Tilde{\pi}(1, y) = y$ for each $y \in \beta X$.

    Since $\Tilde{\pi}$ is separately continuous, there is an open neighbourhood $N$ of 1, such that $\Tilde{\pi}(N, z) \subset V$ and $\Tilde{\pi}(N, x) \subset W$. Clearly, $\Tilde{f} \circ \Tilde{\pi}$ is a separately continuous extension of $f \circ \pi$, so by Theorem \ref{contCriterion}, there is a point $r \in N \cap S$ such that $f \circ \pi$ is continuous at each point $(r, y)$. Since $S$ is right topological, $r u_\gamma \to r1 = r$. Since $f \circ \pi$ is continuous at $(r, x)$, we also have that $f(\pi(r u_\gamma, x_\gamma)) \to f(\pi(r, x)) \geq \frac{1}{2}$, since $\pi(r, x) \in W$.

    On the other hand, since $r, u_\gamma \in S$ and $x_\gamma \in X$, we have:
    \begin{center}
        $\Tilde{f}(\Tilde{\pi}(r u_\gamma, x_\gamma)) = f(\pi(r u_\gamma, x_\gamma)) = f(\pi(r, \pi(u_\gamma, x_\gamma))) = \Tilde{f}((\Tilde{\pi}(r, \Tilde{\pi}(u_\gamma, x_\gamma)))$.
    \end{center}

    But $\Tilde{\pi}(u_\gamma, x_\gamma) \to z$, whence, due to separate continuity, $\Tilde{\pi}(r, \Tilde{\pi}(u_\gamma, x_\gamma))$ converges to $\Tilde{\pi}(r, z)$ and $\Tilde{f}(\Tilde{\pi}(r, \Tilde{\pi}(u_\gamma, x_\gamma)))$ converges to $\Tilde{f}(\Tilde{\pi}(r, z)) = 0$, since $\Tilde{\pi}(r, z) \in X \setminus U$. The resulting contradiction finishes the proof.
\end{proof}

\begin{corollary}\label{1contTheoremExt}
  {\it  Let $S$ be a pseudocompact right topological monoid, $X$ a pseudocompact space, $\pi : S \times X \to X$ a separately continuous action such that $\pi(1, x) = x$ for each $x \in X$ and $\pi$ extends to a separately continuous function $\Tilde{\pi} : \beta S \times \beta X \to \beta X$. Then $\pi$ is continuous at each point of $\lbrace 1 \rbrace \times X$.}
\end{corollary}

\begin{definition}
    A pair of topological spaces $(X, Y)$ is called a \textit{Grothendieck pair}, if for each continuous map $f : X \to C_p(Y)$ space $\overline{f(X)}$ is compact.
\end{definition}

Notion of a Grothendieck pair was introduced by Reznichenko in \cite{ReznichenkoOld}. It is easy to see that if $(X, Y)$ is a Grothendieck pair then $X$ is necessarily pseudocompact.

\begin{theorem}(Reznichenko)\label{extTheorem}
    Let $X$, $Y$, $Z$ be Tychonoff topological spaces, $X$ and $Y$ pseudocompact, $(X, Y)$ a Grothendieck pair and $f : X \times Y \rightarrow Z$ is separately continuous. Then $f$ can be extended to a separately continuous function $\Tilde{f} : \beta X \times \beta Y \rightarrow \beta Z$.
\end{theorem}

\begin{corollary}\label{invCont}
 {\it Let $S$ be a pseudocompact right topological monoid, $X$ a pseudocompact space and $\pi : S \times X \to X$ a separately continuous and quasicontinuous action. Then $\pi$ is jointly continuous at each point of $\lbrace u \rbrace \times X$, where $u$ is a unit of $S$.}
\end{corollary}

\begin{proof}
    The proof is almost the same as the proof of Corollary 5.2 in \cite{LawsonSep}. We show first that $\pi(S, \pi(1, X)) \subset \pi(1, X)$. Indeed, if $x \in \pi(S, \pi(1, X))$, then $x = \pi(s, \pi(1, y))$ for some $s \in S$ and $y \in X$. But $\pi(s, \pi(1, y)) = \pi(1s, \pi(1, y)) = \pi(1, \pi(s, \pi(1, y))) = \pi(1, \pi(s1, y)) \in \break \in \pi(1, X)$.

    Let $\Tilde{\pi} : \beta S \times \beta X \to \beta X$ be a separately continuous extension of $\pi$, $Y = \overline{\pi(1, X)}^{\beta X}$ and $\pi' = \Tilde{\pi}\restriction_{S \times Y}$. We show that $\pi'$ is a quasicontinuous action with range in $Y$ and $\pi'(1, y) = y$ for each $y \in Y$. Indeed, take any $s_1, s_2 \in S$ and $y \in Y$ and let $y_\alpha$ be a net in $\pi(1, X)$ converging to $y$. Then $\pi'(s_1s_2, y_\alpha)$ converges to $\pi'(s_1s_2, y)$ due to separate continuity of $\pi'$. But since $y_\alpha \in \pi(1, X) \subset X$, we also have that $\pi'(s_1s_2, y_\alpha) = \pi(s_1s_2, y_\alpha) = \pi(s_1, \pi(s_2, y_\alpha)) = \pi'(s_1, \pi'(s_2, y_\alpha)) \rightarrow \pi'(s_1, \pi'(s_2, y))$ (we apply separate continuity of $\pi'$ twice here), so $\pi'(s_1s_2, y) = \pi'(s_1, \pi'(s_2, y))$.

    Note that $\pi(1, \pi(1, x)) = \pi(1 \cdot 1, x) = \pi(1, x)$, so it is easy to see that $\pi'(1, y) = y$ for each $y \in Y$.

    Now we show that $\pi'(S, Y) \subset Y$. Indeed, since $\pi'$ is separately continuous, we have that $\pi'(S, Y) = \bigcup\limits_{s \in S} \pi'(s, Y) = \bigcup\limits_{s \in S} \pi'(s, \overline{\pi(1, X)}^{\beta X}) \subset \bigcup\limits_{s \in S} \overline{\pi'(s, \pi(1, X))}^{\beta X} = \bigcup\limits_{s \in S} \overline{\pi(s, \pi(1, X))}^{\beta X} \subset \bigcup\limits_{s \in S} \overline{\pi(1, X)}^{\beta X} = Y$.

    Finally, since $S$ is pseudocompact and $Y$ is compact, $(S, Y)$ is a Grothendieck pair by \cite{}, so Theorems \ref{extTheorem} and \ref{contCriterion1} show that $\pi'$ is quasicontinuous.

    Previous considerations show that $\pi'$ is jointly continuous at each point of $\lbrace 1 \rbrace \times \pi(1, X)$ by Theorem \ref{1contTheorem}. Let $f(s, x) = (u^{-1}s, \pi(1, x))$ and $g(x) = \pi(u, x)$, where $s \in S$ and $x \in X$. Then $g \circ \pi' \circ f = \pi$. Indeed, we have \begin{center}
    $g \circ \pi' \circ f(s, x) = g \circ \pi'(u^{-1}s, \pi(1, x)) = g(\pi'(u^{-1}s, \pi(1, x)) = \pi(u, \pi(u^{-1}s, \pi(1, x)) = \pi(uu^{-1}s1, x) = \pi(s, x)$.\end{center} On the other hand, $f$ is continuous as a product of continuous maps and $f(u, x) = (u^{-1}u, \pi(1, x)) = (1, \pi(1, x))$, so $\pi'$ is continuous at $f(u, x)$. Finally, $g$ is continuous due to separate continuity of $\pi$, so $\pi$ is continuous at $(u, x)$.
\end{proof}

\begin{corollary}
  {\it  Let $S$ be a right topological monoid, $X$ a pseudocompact space, $(S, X)$ a Grothendieck pair and $\pi : S \times X \to X$ a separately continuous action. Then $\pi$ is jointly continuous at each point of $\lbrace u \rbrace \times X$, where $u$ is a unit of $S$.}
\end{corollary}

In conclusion we consider a simple example, demonstrating that Corollary \ref{invCont} cannot be strengthened to the joint continuity of the multiplication at every point of a compact semitopological semigroup.

\begin{example}
Let $X = \mathbb{R} \cup \lbrace \alpha \rbrace$  be a one-point compactification of $\mathbb{R}$, so the sets $\lbrace \alpha \rbrace \cup (\mathbb{R} \setminus [-x, x])$, where $x > 0$ form a local base for $\alpha$. Define multiplication on $X$ as follows:

\[
    xy =
\begin{cases}
    x + y,& \text{if } x, y \in \mathbb{R}\\
    \alpha,& \text{if } x = \alpha \text{ or } y = \alpha.
\end{cases}
\]

Note that $X$ is commutative. Let us show that $X$ is a semitopological semigroup. Let $s_x : X \to X$ be a map such that $s_x(y) = xy$. Note that $s_\alpha$ is constant, and, therefore, continuous. Moreover, the map $s_x$ is jointly continuous at $y \in \mathbb{R}$, since $\mathbb{R}$ is a topological subsemigroup of $X$ and open sets of $\mathbb{R}$ are also open in $X$. So it remains to check the continuity of $s_x$ at $\alpha$ for each $x \in \mathbb{R}$. Indeed, let $U = \mathbb{R} \setminus [-a, a]$ be an arbitrary basic open neighbourhood of $\alpha$. Consider $V = \lbrace \alpha \rbrace \cup (\mathbb{R} \setminus [-a - 2|x|, a + 2|x|]$. Then
\begin{center}
$X \setminus xV = X \setminus (\lbrace \alpha \rbrace \cup (x + \mathbb{R} \setminus [-a - 2|x|, a + 2|x|])) = \mathbb{R} \setminus ((x + \mathbb{R} \setminus [-a - 2|x|, a + 2|x|]))) = \mathbb{R} \setminus (\mathbb{R} \setminus (x + [-a - 2|x|, a + 2|x|])) = x + [-a - 2|x|, a + 2|x|] \supset [-a, a] = X \setminus U$,
\end{center}
so $xV \subset U$.

Note that 0 is an identity for $X$ and $\alpha$ is the only element of $X$ which is not a unit. We are going to show that the multiplication is discontinuous at $(\alpha, \alpha)$. Indeed, since any basic open neighbourhood $U$ of $\alpha$ has inverse elements, set $UU$ always contains a 0, but any basic neighbourhood $V$ of $\alpha$ doesn't, so $UU \not\subset V$, which means that the multiplication is discontinuous at $(\alpha, \alpha)$.
\end{example}

In \cite{Korovin}, A. Korovin presented a method for the construction of
pseudocompact semitopological groups that are neither topological groups nor
homeomorphic to any topological group. Moreover, pseudocompact Korovin's orbits are not homeomorphic to any Mal'tsev space \cite{RT}.

\begin{example}(Korovin's orbits)  {\it There is a pseudocompact semitopological group $G$ such that $G$ is not homeomorphic to any topological group.}
\end{example}

As the Korovin's example shows, the condition of pseudocompactness is not enough for the main theorem to be valid.

\section{Weak $q_D$-spaces}

Let us recall some definitions.

\begin{definition}\label{weakqdDef}
Let $X$ be a topological space, $x \in X$ and $A \subset X, A \neq \overline{A}$ and $\mathcal{N}(x)^{<\omega}$ is a set of all finite sequences of elements of $\mathcal{N}(x)$. The point $x$ is called a \textit{weak q-point} relative to $A$, if there is a function $\varphi: \mathcal{N}(x)^{<\omega} \rightarrow \mathcal{P}(A)$ such that:
\begin{enumerate}
    \item For each $\varphi(\langle U_1, \ldots, U_n \rangle)$ there is $T \subset X, |T| \leq \aleph_0$ such that \\$\varphi(\langle U_1, \ldots, U_n \rangle) \subset \overline{T}$;
    \item For each sequence $\lbrace U_i \rbrace_{i=1}^\infty$ of open neighbourhoods of $x$ the set $\overline{\bigcup\limits_{i=1}^\infty \varphi(\langle U_1, \ldots, U_i \rangle)} \cap \bigcap\limits_{i=1}^\infty U_i$ is not empty.
\end{enumerate}
\end{definition}

\begin{definition}
    A topological space $X$ is called a \textit{weak $q$-space}, if for each non-closed subset $A$ of $X$ there exists a weak $q$-point relative to $A$.
\end{definition}

Notions of a weak $q$-point and a weak $q$-space were introduced by M.O.~Asanov and N.V.~Velichko in ~\cite{AsanovVel}. In the same paper it was shown that the class of weak $q$-spaces includes such well-known classes as: $q$-spaces, $k$-spaces, spaces of countable tightness and locally separable spaces. In this section we introduce a somewhat more general notion of a \textit{weak $q_D$-space} and study some of its properties.

\begin{definition}
    Let $X$ be a topological space, $D \subset X, \overline{D} = X$. Space $X$ is called a \textit{weak $q_D$-space}, if each non-closed in $X$ subset of $D$ has a weak $q$-point.
\end{definition}

Obviously, weak $q$-spaces are also weak $q_D$-spaces. Also, introducing the notion of weak $q_D$-space is motivated by the fact that they include such a wide class as $q_D$-spaces.

\begin{definition}
    Let $X$ be a topological space. $X$ is called a \textit{$q_D$-space}, if there is a dense subset $D \subset X$ such that each point $x \in X$ has a countable family of its open neighbourhoods $\lbrace U_i \rbrace_{i=1}^\infty$ such that for each choice of $d_i \in U_i \cap D$ the set $\lbrace d_i \rbrace_{i=1}^\infty$ has a cluster point.
\end{definition}

The notion of a $q_D$-point was first introduced by P.S.~Kenderov, I.S.~Kortezov and W.B.~Moors in the article \cite{KKM}. Some generalisations of $q_D$-points, involving feeble compactness or pseudocompactness, have also been studied in \cite{CKM} and \cite{M}.

\begin{proposition}
   Any regular $q_D$-space is a weak $q_D$-space.
\end{proposition}

\begin{proof}
    Let $X$ be a regular $q_D$-space and $A \subset D$ not closed. Take an arbitrary point $x \in \overline{A} \setminus A$. Let us show that $x$ is a weak $q$-point relative to $A$. Since $X$ is a $q_D$-space, there is a countable family $\lbrace U_i \rbrace_{i=1}^\infty$ of open neighbourhoods of $x$ such that for each $d_i \in U_i \cap D$ the set $\lbrace d_i \rbrace_{i=1}^\infty$ has a cluster point. We define the function $\psi : \mathcal{N}(x)^{<\omega} \rightarrow \mathcal{N}(x)$ by recursion as follows:

    $\psi(\langle W_1 \rangle) = O$, where $O$ is an open neighbourhood of $x$, such that $\overline{O} \subset W \cap U_1$.

    $\psi(\langle W_1, \ldots, W_n \rangle) = O$, where $O$ is an open neighbourhood of $x$, such that

    $\overline{O} \subset \psi(\langle W_1, \ldots, W_{n-1} \rangle) \cap U_n \cap W_n$.

     For each $\langle W_1, \ldots, W_n \rangle \in \mathcal{N}(x)^{<\omega}$ fix a point $\varphi(\langle W_1, \ldots, W_n \rangle)$ in $\psi(\langle W_1, \ldots, W_n \rangle) \cap A$. Let us show that $\varphi$ satisfies conditions 1 and 2 from Definition \ref{weakqdDef}.

    Indeed, condition 1 from Definition \ref{weakqdDef} is obviously satisfied. Let us check the second condition. Let $\lbrace V_i \rbrace_{i=1}^\infty$ be arbitrary sequence of an open neighbourhoods of $x$. Since for each $i$ we have $\varphi(\langle V_1, \ldots, V_i \rangle) \in U_i \cap D$, the set $\lbrace \varphi(\langle V_1, \ldots, V_i \rangle), i \in \mathbb{N} \rbrace$ has a cluster point $d$. In particular, $d \in \overline{\bigcup\limits_{i=1}^\infty \varphi(\langle V_1, \ldots, V_i\rangle)}$. We are going to show that $d \in \bigcap\limits_{i=1}^\infty V_i$. Suppose the opposite and take $j \in \omega$, such that $d \not\in V_j$. Then $d \not\in \psi(\langle V_1, \ldots, V_j\rangle)$. But $\overline{\psi(\langle V_1, \ldots, V_k\rangle)} \subset \psi(\langle V_1, \ldots, V_j\rangle)$ for $k > j$. Take an open neighbourhood $W$ of $d$ such that $W \cap \psi(\langle V_1, \ldots, V_{j+1}\rangle) = \varnothing$ and $\varphi(\langle V_1, \ldots, V_l \rangle) \not\in W$ for $1 \leq l \leq j$. Then $W \cap \lbrace \varphi(\langle V_1, \ldots, V_i \rangle), i \in \mathbb{N} \rbrace = \varnothing$, so $d$ is not a cluster point for $\lbrace \varphi(\langle V_1, \ldots, V_i \rangle), i \in \mathbb{N} \rbrace$. The resulting contradiction completes the proof.
\end{proof}

Recall that a subset $A$ of $X$ is called {\it regular closed} set if it is equal to the closure of its interior.

\begin{proposition}\label{canonClosedqd}
    Let $X$ be a weak $q_D$-space and $A$ a regular closed subset of $X$. Then $A$ is a weak $q_D$-space.
\end{proposition}

\begin{proof}
    Let $A = \overline{U}$, where $U$ is open in $X$. Obviously, $D \cap U$ is dense in $A$. If $B \subset D \cap U$ is not closed in $A$, then it is not closed in $X$ and, therefore, has a weak $q$-point $x$. It is easy to see that $x \in \overline{U} = A$. We claim that $x$ is a weak $q$-point relative to $B$ in $A$. Let $\varphi$ be a function from the definition of a weak $q$-point for $X$. For each open neighbourhood $U$ of $x$ in $A$ fix an open set $\Tilde{U}$ in $X$ such that $\Tilde{U} \cap A = U$. Let's define a function $\Tilde{\varphi}$ as following: $\Tilde{\varphi}(\langle U_1, \ldots, U_n \rangle) = \varphi(\langle \Tilde{U_1}, \ldots, \Tilde{U_n} \rangle)$ and show that $\Tilde{\varphi}$ is the desired function.

    Let us check that condition 1 from Definition \ref{weakqdDef} is satisfied. Take an arbitrary finite sequence $U_1, \ldots, U_n$ of open neighbourhoods of $x$ in $A$ and a countable set $T \subset X$ such that $\Tilde{\varphi}(\langle U_1, \ldots, U_n \rangle) = \varphi(\langle \Tilde{U_1}, \ldots, \Tilde{U_n} \rangle) \subset \overline{T}$. Let us show that $\Tilde{\varphi}(\langle U_1, \ldots, U_n \rangle) \subset \overline{T \cap U}$. Suppose that there is $t \in \Tilde{\varphi}(\langle U_1, \ldots, U_n \rangle) \setminus \overline{T \cap U}$. Then there is an open neighbourhood $V$ of $t$ in $A$, such that $V \cap (T \cap U) = \varnothing$. Since $t \in U$, assume that $V$ is open in $X$ and $V \subset U$. But then $V \cap (T \cap U) = (V \cap U) \cap T = V \cap T = \varnothing$, which contradicts the fact that $t \in \overline{T}$, so $\Tilde{\varphi}(\langle U_1, \ldots, U_n \rangle) \subset \overline{T \cap U}$. It remains to note that $T \cap U$ is countable and the closure of $T \cap U$ in $X$ coincides with its closure in $A$, since $A$ is closed $X$. Thus, condition 1 of Definition \ref{weakqdDef} is satisfied.

    Now we show that condition 2 of Definition \ref{weakqdDef} is also satisfied for $\Tilde{\varphi}$. Take an arbitrary sequence $\lbrace U_i \rbrace_{i=1}^\infty$ of open neighbourhoods of $x$ in $A$. Then $\overline{\bigcup\limits_{i=1}^\infty \varphi(\langle \Tilde{U_1}, \ldots, \Tilde{U_i} \rangle)} \cap \bigcap\limits_{i=1}^\infty \Tilde{U_i} \neq \varnothing$. But $\overline{\bigcup\limits_{i=1}^\infty \varphi(\langle \Tilde{U_1}, \ldots, \Tilde{U_i} \rangle)} \subset A$, so
    \begin{center}
       $\overline{\bigcup\limits_{i=1}^\infty \varphi(\langle \Tilde{U_1}, \ldots, \Tilde{U_i} \rangle)} \cap \bigcap\limits_{i=1}^\infty \Tilde{U_i} = (\overline{\bigcup\limits_{i=1}^\infty \varphi(\langle \Tilde{U_1}, \ldots, \Tilde{U_i} \rangle)} \cap A) \cap \bigcap\limits_{i=1}^\infty \Tilde{U_i} = \overline{\bigcup\limits_{i=1}^\infty \varphi(\langle \Tilde{U_1}, \ldots, \Tilde{U_i} \rangle)} \cap (\bigcap\limits_{i=1}^\infty \Tilde{U_i} \cap A) = \overline{\bigcup\limits_{i=1}^\infty \varphi(\langle \Tilde{U_1}, \ldots, \Tilde{U_i} \rangle)} \cap \bigcap\limits_{i=1}^\infty (\Tilde{U_i} \cap A) = \overline{\bigcup\limits_{i=1}^\infty \Tilde{\varphi}(\langle U_1, \ldots, U_i \rangle)} \cap \bigcap\limits_{i=1}^\infty U_i \neq \varnothing$
    \end{center}
    That is, condition 2 of Definition \ref{weakqdDef} is also satisfied for $\Tilde{\varphi}$.
\end{proof}

Recall that a subset $A$ of a topological space $X$ is called {\it bounded} if every continuous function on $X$ is bounded on $A$.

\medskip

The following proposition was proved in \cite{AsanovVel}.

\begin{proposition} \label{predcompLemma}
If $F$ is bounded in $C_p(X)$, then for each separable $Y \subset X$ there is $F_1 \subset C_p(X)$ such that $F \subset F_1$ and $\pi_Y(F_1) = \overline{\pi_Y(F)}^{C_p(Y)}$ is compact and metrizable where $\pi_Y(f) = f\restriction_Y$.
\end{proposition}

\begin{corollary}\label{restrLemma}
{\it If a set $F$ is bounded in $C_p(X)$ and $f \in \overline{F}^{\mathbb{R}^X}$, then for each separable $Y \subset X$ the restriction $f\restriction_Y$ is continuous.}
\end{corollary}

\begin{proof}
    Since the map $\pi_Y: \mathbb{R}^X \rightarrow \mathbb{R}^Y$ is continuous, $f\restriction_Y \in \overline{\pi_Y(F)}^{\mathbb{R}^Y}$. On the other hand, since the set $\overline{\pi_Y(F)}^{C_p(Y)}$ is compact, it is closed in $\mathbb{R}^Y$ and since $\pi_Y(F) \subset \overline{\pi_Y(F)}^{C_p(Y)} \subset \overline{\pi_Y(F)}^{\mathbb{R}^Y}$ we have $\overline{\pi_Y(F)}^{C_p(Y)} = \overline{\pi_Y(F)}^{\mathbb{R}^Y}$, whence $f\restriction_Y \in C_p(Y)$.
\end{proof}

 In \cite{AsanovVel} it was proved that if $X$ is a pseudocompact weak $q$-space, then $(X, X)$ is a Grothendieck pair. We show that a similar theorem also holds for $q_D$-spaces, even though the idea is very much the same.

\begin{theorem}\label{weakqd}
Let $X$ be a pseudocompact space, $Y$ a weak $q_D$-space. Then $(X, Y)$ is a Grothendieck pair.
\end{theorem}

\begin{proof}
Take an arbitrary continuous map $\Phi : X \rightarrow C_p(Y)$. Then $\Phi(X)$ is also pseudocompact and, therefore, bounded. We need to show that $F = \overline{\Phi(X)}$ is compact. Since $C_p(Y) \subset \mathbb{R}^Y$ and $\mathbb{R}^Y$ is Diedonne-complete, $\overline{\Phi(X)}^{\mathbb{R}^Y}$ is compact, so it is enough to show that $F = \overline{\Phi(X)}^{\mathbb{R}^Y}$.

Let $D \subset X$ be a dense set from the definition of a weak $q_D$-space. Take an arbitrary point $f \in \overline{\Phi(X)}^{\mathbb{R}^Y}$ suppose that $f \not\in C_p(Y)$, that is, $f$ is discontinuous. Due to a proposition proved in \cite{ArkhPon}, there is a point $x \in Y$ such that $f\restriction_{D \cup \lbrace x \rbrace}$ is discontinuous. Without a loss of generality we can assume that $x \in D$. Then there is a closed $A \subset \mathbb{R}$ such that $B = (f\restriction_D)^{-1}(A) = f^{-1}(A) \cap D$ is not closed in $D$. Then $B$ is not closed in $X$. Let $x_0 \in \overline{B} \setminus B$. Since $f(x_0) \not\in A$ and $A$ is closed, $\rho(f(x_0), A) = \varepsilon > 0$. Fix a function $\varphi$ from the definition of a weak $q$-point. Using the induction, we find a sequence $\lbrace f_i \rbrace_{i=1}^\infty \subset \Phi(X)$, sequences of sets $\lbrace C_i \rbrace_{i=1}^\infty \subset \mathcal{P}(A)$ and $\lbrace E_i \rbrace_{i=1}^\infty, E_i = \lbrace x_i^k \rbrace$ and a sequence of open sets $\lbrace W_i \rbrace_{i=1}^\infty \subset \mathcal{N}(x_0)$, satisfying the following conditions:

1) $|f(x_0) - f_i(x_0)| < \frac{\varepsilon}{4}$;

2) $|f(x_0) - f_i(x)| < \frac{\varepsilon}{4}$ for each $x \in W_i$;

3)  $|E_i| \leq \aleph_0$ and $\overline{E_i} \supset C_i$;

4) $|f_i(x_j^k) - f(x_0)| > \frac{3\varepsilon}{4}$ for each $j, k < i$ such that $|f(x_j^k) - f(x_0)| > \frac{3\varepsilon}{4}$.

\textit{The base of induction.} Since $f \in \overline{\Phi(X)}^{\mathbb{R}^X}$, there is a function $f_1 \in [f, x_0, \frac{\varepsilon}{4}] \cap \Phi(X)$. For $W_1$ we take an open set $f_1^{-1}(B(f(x_0), \frac{\varepsilon}{4})$, for $C_1$ a set $\varphi(\lbrace W_1 \rbrace)$ and for $E_1$ we take a countable subset of $X$, such that $C_1 \subset \overline{E_1}$. Condition 4) on the first step is satisfied automatically.

\textit{The step of induction.} Let functions $f_1, \ldots, f_k$ and sets $W_1, \ldots, W_k$, $C_1, \ldots, C_k$, $E_1, \ldots, E_k$, satisfying 1)-4) be constructed. Let $H_{k+1} = \lbrace x_m^n : m, n < k+1, |f(x_m^n) - f(x_0)| > \frac{3\varepsilon}{4}\rbrace$, $\varepsilon_1 = \min{\lbrace |f(x) - f(x_0)| - \frac{3\varepsilon}{4}, x \in H_{k+1} \rbrace}$ and $\varepsilon_2 = \min\lbrace \frac{\varepsilon_1}{2}, \frac{\varepsilon}{4} \rbrace$. Take $f_{k+1} \in [f, \lbrace x_0 \rbrace \cup H_k, \varepsilon_2] \cap \Phi(x)$. Condition 1) for $f_{k+1}$ is obviously satisfied. Let us check condition 2). Due to the choice of $\varepsilon_2$ we have
\begin{align*}
  |f(x) - f(x_0)| > \varepsilon_2 + \frac{3\varepsilon}{4}
\end{align*}
 for each $x \in H_{k+1}$. Then
 \begin{center}
     $\frac{3\varepsilon}{4} +\varepsilon_2 < |f(x) - f(x_0)| \leq |f(x) - f_{k+1}(x)| + |f_{k+1}(x) - f(x_0)| < \newline < \varepsilon_2 + |f_{k+1}(x) - f(x_0)|$,
 \end{center}
 whence it follows that
 \begin{align*}
     |f_{k+1}(x) - f(x_0)| > \frac{3\varepsilon}{4}.
 \end{align*}
 Let $W_{k+1} = f_{k+1}^{-1}(B(f(x_0), \frac{\varepsilon}{4})) \cap \bigcap_{j=1}^{k+1} W_j$, $C_{k+1} = \varphi(\lbrace W_1, \ldots, W_{k+1}\rbrace)$ and take $E_{k+1}$ such that $C_{k+1} \subset \overline{E_{k+1}}$.

 Take $y \in \overline{\bigcup\limits_{i=1}^\infty C_{i}} \cap \bigcap\limits_{i=1}^\infty W_i$ and $T = \lbrace x \in \bigcup\limits_{i=1}^\infty E_i : |f(x) - f(x_0)| > \frac{3\varepsilon}{4}\rbrace$. Let us show that $y \in \overline{T}$. Take $U \in \mathcal{N}(y)$. Since $y \in \overline{\bigcup\limits_{i=1}^\infty C_{i}}$, there is $j \in \omega$ such that $U \cap C_j \neq \varnothing$. Take $y' \in U \cap C_j$. Since the space $M = E_j \cup \lbrace y' \rbrace$ is separable, $f\restriction_M$ is continuous by corollary \ref{restrLemma}. Since $f(y') \in B$ and $f(x_0) \not\in B$, $|f(x_0) - f(y')| \geq \varepsilon$. Since $y' \in C_j \subset \overline{E_j}$ and the restriction $f\restriction_M$ is continuous, the set $(f\restriction_M)^{-1}(B(f(y'), \frac{\varepsilon}{5})) \cap U = f^{-1}(B(f(y'), \frac{\varepsilon}{4})) \cap M \cap U$ is an open neighbourhood of $y'$ in $M$ and $f^{-1}(B(f(y'), \frac{\varepsilon}{4})) \cap M \cap U \cap E_j \neq \varnothing$. Take $y'' \in f^{-1}(B(f(y'), \frac{\varepsilon}{5})) \cap M \cap U \cap E_j$. Then
 \begin{center}
     $\varepsilon \leq |f(x_0) - f(y')| \leq |f(x_0) - f(y'')| + |f(y'') - f(y')| < |f(x_0) - f(y'')| + \frac{\varepsilon}{4}$,
 \end{center}
 whence $|f(x_0) - f(y'')| > \frac{3\varepsilon}{4}$, i. e. $y'' \in U \cap T$.

 Let $Z = T \cup \lbrace y, x_0 \rbrace$. Since $Z$ is countable and $\lbrace f_i \rbrace_{i=1}^\infty$ is bounded as a subset of $F$, $\overline{\pi_Z(\lbrace f_i \rbrace_{i=1}^\infty)}^{C_p(Z)}$ is compact by lemma \ref{predcompLemma}. In particular, it means that $\pi_Z(\lbrace f_i \rbrace_{i=1}^\infty)$ has a cluster point $g \in C_p(Z)$.

 Let us show that $|g(y) - f(x_0)| \geq \frac{\varepsilon}{2}$. Indeed, if $t \in T$, then $|f_n(t) - f(x_0)| > \frac{3\varepsilon}{4}$ for $n$ large enough. Since $[g, t, \frac{\varepsilon}{4}] \cap \lbrace f_i \rbrace_{i=1}^\infty$ is infinite, there is $m$ such that inequalities
 \begin{gather*}
    |f_m(t) - f(x_0)| > \frac{3\varepsilon}{4},\\
    |f_m(t) - g(t)| < \frac{\varepsilon}{4}
 \end{gather*}
are satisfied simultaneously. Hence
\begin{center}
    $\frac{3\varepsilon}{4} < |f_m(t) - f(x_0)| \leq |f_m(t) - g(t)| + |g(t) - f(x_0)| < \frac{\varepsilon}{4} + |g(t) - f(x_0)|$
\end{center}
and $|g(t) - f(x_0)| > \frac{\varepsilon}{2}$.
Since $y \in \overline{T}$ and $g\restriction_Z$ is continuous, we also have $g(y) \in g(\overline{T}) \subset \overline{g(T)} \subset \overline{\mathbb{R} \setminus B[f(x_0), \frac{\varepsilon}{2}]} = \mathbb{R} \setminus B(f(x_0), \frac{\varepsilon}{2})$, that is, $|g(y) - f(x_0)| \geq \frac{\varepsilon}{2}$.

On the other hand, $y \in \bigcap_{i=1}^\infty W_i$, which means that $|f_k(y) - f(x_0)| < \frac{\varepsilon}{4}$ for each $k$. Take $f_k \in [g, y, \frac{\varepsilon}{12}] \cap \lbrace f_i \rbrace_{i=1}^\infty$. Then
\begin{center}
    $|g(y) - f(x_0)| \leq |g(y) - f_k(y)| + |f_k(y) - f(x_0)| < \frac{\varepsilon}{4} + \frac{\varepsilon}{12} = \frac{\varepsilon}{3}$
\end{center}
The obtained contradiction shows that $f \in C_p(Y)$ and $F$ is compact, as required.
\end{proof}

\begin{corollary}\label{monoidContCor}
{\it Let $S$ be a pseudocompact weak $q_D$-space and a semitopological monoid. Then its multiplication is continuous at each point of $\lbrace u \rbrace \times S$ where $u$ is a unit of $S$.}
\end{corollary}
\begin{proof}
    Immediately follows from Theorems \ref{1contTheorem} and \ref{weakqd} and the fact that the multiplication is an action of $S$ on itself.
\end{proof}

\begin{corollary}\label{monoidCont}
  {\it  Let $S$ be a pseudocompact weak $q_D$-space and a semitopological semigroup, $S'$ a semi-open subsemigroup of $S$ with an identity $e$. Then the multiplication restricted to $S' \times S$ is continuous at each point of $In(S') \times S$ where $In(S')$ is a set of all units of $S'$.}
\end{corollary}

\begin{proof}
    Recall that the closure of a subsemigroup of a semitopological semigroup is again a subsemigroup. Since $e$ is the identity for $S'$ and $S'$ is dense in $\overline{S'}$, it is easy to see that $e$ is also the identity for $\overline{S'}$. Since $S'$ is semi-open, $\overline{S'}$ is regular closed. Then $\overline{S'}$ is a pseudocompact and by Propostition \ref{canonClosedqd} it is also a weak $q_D$-space. $(\overline{S'}, S)$ is a Grothendieck pair by Theorem \ref{weakqd}. It remains to note that the restriction of the multiplication $\varphi$ on $\overline{S'} \times S$ is a separately continuous action of $\overline{S'}$ on $S$. This corollary now follows from Corollary \ref{invCont}.
\end{proof}

\begin{corollary}\label{groupRestr}
  {\it  Let $S$ be a pseudocompact weak $q_D$-space and a semitopological semigroup and $G$ a semi-open subgroup of $S$. Then the restriction of the multiplication to $G \times S$ is continuous.}
\end{corollary}

\begin{proof}
    Follows from Corollary \ref{monoidCont}.
\end{proof}

\begin{corollary}\label{paratopCor}
   {\it Let $S$ be a pseudocompact, weak $q_D$-space and a semitopological semigroup and $G$ a semi-open subgroup of $S$. Then $G$ is a paratopological group.}
\end{corollary}

The proof of the following statement is only slightly different from the proof of Theorem 2.4.1 in \cite{Tgrs}.
\begin{theorem}\label{paratopTheorem}
    Let $G$ be a paratopological group and $G \subset X$ where $X$ is a regular pseudocompact space and $\overline{\operatorname{Int}{G}} = X$. Then $G$ is a topological group.
\end{theorem}

\begin{proposition}
    Let $S$ be a pseudocompact, weak $q_D$-space and a semitopological semigroup and $G$ a semi-open subgroup of $S$. Then $G$ is a topological group.
\end{proposition}
\begin{proof}
By Corollary \ref{paratopCor},  $G$ is a paratopological group. Since $G$ is semi-open in $S$, $\operatorname{Int}{G}$ is dense in $\overline{G}$. Then $G$ is a topological group by Theorem \ref{paratopTheorem}.
\end{proof}

\section{Subgroups of semitopological semigroups}

Recall that a topologized group $G$ is called \textit{paratopological}, if the multiplication $(x, y) \mapsto xy$ is continuous. If, moreover, taking the inverse $x \mapsto x^{-1}$ is also continuous, then $G$ is called a \textit{topological} group.

In this section we show that, under certain additional conditions, subgroups of pseudocompact semitopological semigroups are actually topological.

\begin{corollary}
   {\it Let $S$ be a semitopological monoid, $(S, S)$ a Grothendieck pair and $G$ a subgroup of $S$. Then $G$ is a topological group.}
\end{corollary}

\begin{proof}
    By Corollary \ref{invCont}, $G$ is a paratopological group, so we only have to prove that taking the inverse is continuous. Consider a function $\varphi' = \Tilde{\varphi}\restriction_{S \times \beta S}$, where $\Tilde{\varphi} : \beta S \times \beta S \to \beta S$ is a separately continuous extension of the multiplication in $S$. We claim that $\varphi'$ is an action of $S$ on $\beta S$. Indeed, fix any $s, t \in S, x \in \beta S$ and let $x_\alpha$ be a net in $S$ converging to $x$. Then, by separate continuity, $\varphi'(st, x_\alpha)$ is converging to $\varphi'(st, x)$. But since $x_\alpha \in S$, we also have that $\varphi'(st, x_\alpha) = (st)x_\alpha = s(tx_\alpha) = \varphi'(s, \varphi'(t, x_\alpha))$. Applying separate continuity of $\varphi'$ twice, we see that $\varphi'(s, \varphi'(t, x_\alpha))$ is converging to $\varphi'(s, \varphi'(t, x))$, so $\varphi'(st, x)$ = $\varphi'(s, \varphi'(t, x))$ and $\varphi'$ is an action.

    Take any $g \in G$ and let $g_\alpha$ be any net in $G$ that converges to $g$. Then $g_\alpha^{-1}$ clusters to some $h \in \beta S$. By Theorem \ref{invCont}, $\varphi'$ is continuous at $(g, h)$, so $g_\alpha g^{-1}_\alpha$ clusters to $\varphi'(g, h) = e$. But $\varphi'(g^{-1}, \varphi'(g, h)) = \varphi'(g^{-1}g, h) = \varphi'(e, h) = h$ and $\varphi(g^{-1}, e) = g^{-1}$, so $h = g^{-1}$.
\end{proof}

\begin{corollary}
   {\it Let $(S, \bullet)$ be a pseudocompact semitopological monoid with quasicontinuous multiplication $\bullet$  and $G$ a subgroup of $S$. Then $G$ is a topological group.}
\end{corollary}

\section{Continuity in semitopological semilattices}

Just as in the case of semigroups, we call a partially ordered set endowed with a topology a \textit{topologized poset}.

A subset $A$ of a poset $X$ is called \textit{convex}, if for each $a, b \in A$ and $a \leq c \leq b$ we have that $c \in A$. A topologized poset is called \textit{locally convex}, if it has a base of convex sets.

A subset $C$ of a poset $(X, \leq)$ is called a \textit{chain}, if $(C, \leq)$ is linearly ordered.

A topologized poset $X$ is called \textit{complete}, if for any chain $C \subset X$ we have $\inf{C} \in \overline{C}$ and $\sup{C} \in \overline{C}$.

We recall that a \textit{semilattice} is a commutative semigroup, each element of which is idempotent (i.e. $xx = x$). It is well known that a semilattice $X$ can be viewed as a poset, namely, each semilattice can be partially ordered as following: $x \leq y \leftrightarrow xy = x$ (this order is called a \textit{natural} order on a semilattice). For an element
$x$ of a semilattice $X$ by ${\downarrow}x=\{y\in X: y\leq x\}$ and ${\uparrow}x=\{y\in X: x\leq y\}$ we denote the {\it lower} and {\it upper cones} of $x$ in $X$ respectively. It is easy to see that $xy$ is the greatest lower bound of elements  $x$ and $y$ in terms of the natural order. Conversely, if in poset $(X, \leq)$ each two-element subset has the greatest lower bound, then $X$ with an operation taking their infimum is a semilattice.

We start with the following lemma.

\begin{lemma} \label{policemens}
Let $X$ be a topologized locally convex poset and let $\lbrace a_\alpha \rbrace_{\alpha \in A}$, $\lbrace b_\beta \rbrace_{\beta \in B}$ be nets, converging to some $x \in X$. If there is a net $\lbrace c_\gamma \rbrace_{\gamma \in \Gamma }$, satisfying the following condition:

$\forall \alpha_0 \in A, \beta_0 \in B$ $\exists \gamma_0 \in \Gamma$ $\forall \gamma \geq \gamma_0$ $\exists \alpha \geq \alpha_0, \beta \geq \beta_0 : a_\alpha \leq c_\gamma \leq b_\beta$.

Then $\lbrace c_\gamma \rbrace$ converges to $x$.
\end{lemma}
\begin{proof}
Fix a convex open neighbourhood $U$ of $x$ and find indexes $\alpha_0 \in A, \beta_0 \in B$ such that $a_\alpha, b_\beta \in U$ for each $\alpha \geq \alpha_0, \beta \geq \beta_0$. Then there is an index $\gamma_0 \in \Gamma$ such that, for each $\gamma \geq \gamma_0$ there are $\alpha \geq \alpha_0, \beta \geq \beta_0$ such that $a_\alpha \leq c_\gamma \leq b_\beta$. Since $U$ is convex, these inequalities together with conditions $a_\alpha \in U, b_\beta \in U$ imply that $c_\gamma \in U$, i. e. $\lbrace c_\gamma \rbrace_{\gamma \in \Gamma}$ converges to $x$.
\end{proof}

If nets $\lbrace a_\alpha \rbrace_{\alpha \in A}$ and $\lbrace b_\beta \rbrace_{\beta \in B}$ have a common domain, the cumbersome condition from the previous lemma can be replaced by a much simpler one:
\begin{corollary}{\it
Let $X$ be a topologized locally convex poset, $\lbrace a_\alpha \rbrace_{\alpha \in \Lambda}$ and $\lbrace b_\alpha \rbrace_{\alpha \in \Lambda}$ are the nets converging to $x \in X$. If a net $\lbrace c_\alpha \rbrace_{\alpha \in \Lambda}$ is such that $a_\alpha \leq c_\alpha \leq b_\alpha$ for each $\alpha \in \Lambda$, then $\lbrace c_\alpha \rbrace_{\alpha \in \Lambda}$ also converges to $x$.}
\end{corollary}
We continue the research \cite{KO} by proving the following theorem.
\begin{theorem}\label{locConvTheorem}
Let $X$ be a semitopological locally convex semilattice, ${\downarrow}x$ pseudocompact and the multiplication restricted to ${\downarrow}x$ is quasicontinuous for each $x \in X$. Then the multiplication of $X$ is continuous at each point of the diagonal $X \times X$.
\end{theorem}

\begin{proof}
Let $\lbrace x_\alpha \rbrace_{\alpha \in A}$ and $\lbrace y_\beta \rbrace_{\beta \in B}$ be the nets, converging to $x$. Since the map $s_x : y \mapsto xy$ is continuous, nets $\lbrace x_\alpha x \rbrace_{\alpha \in A}$ and $\lbrace y_\beta x \rbrace_{\beta \in B}$ converge to $xx=x$ and their elements are subsets of ${\downarrow}x$. By Corollary \ref{invCont}, the multiplication, restricted on ${\downarrow}x$, is continuous at $x$, so the net $\lbrace x_\alpha x y_\beta x \rbrace_{(\alpha, \beta) \in A \times B}$ = $\lbrace x_\alpha  y_\beta x \rbrace_{(\alpha, \beta) \in A \times B}$ converges to $x$. Since for each $\alpha \in A, \beta \in B$ we have that $x_\alpha y_\beta x \leq x_\alpha y_\beta \leq x_\alpha$, the net $\lbrace x_\alpha y_\beta \rbrace_{(\alpha, \beta) \in A \times B}$ converges to $x$ by Lemma \ref{policemens}.
\end{proof}

\begin{lemma}\label{GrothPairLemma}
    Let $X$ be a space, $(X, X)$ a Grothendieck pair and $f : X \to Y$ a continuous surjective map. Then $(Y, Y)$ is a Grothendieck pair.
\end{lemma}
\begin{proof}
    Take any continuous map $\Phi : Y \to C_p(X)$. Then $\overline{\Phi(Y)} = \overline{\Phi(f(X))}$, which is compact since $(X, X)$ is a Grothendieck pair, so $(Y, X)$ is also a Grothendieck pair. Then, by \cite{Reznichenko}, $(X, Y)$ is a Grothendieck pair, so, applying the previous reasoning once again, we get that $(Y, Y)$ is a Grothendieck pair.
\end{proof}

For sets $X$  and $Y$ a map $\Phi : X \rightarrow \mathcal{P}(Y)$ is called a \textit{multimap} from $X$ to $Y$ and denoted as $\Phi : X \multimap Y$. An \textit{image} of a set $A \subset X$ by a multimap $\Phi$ is a set $\Phi(A) = \bigcup\limits_{x \in A} \Phi(x)$, and a \textit{preimage} of a set $B \subset Y$ is a set $\Phi^{-1}(B) = \lbrace x \in X : \Phi(x) \cap B \neq \varnothing \rbrace$. A multimap $\Phi : X \multimap Y$ between topological spaces $X$ and $Y$ is called a $T_2$-multimap, if $\Phi(x)$ is $\theta$-closed in $Y$ for each $x \in X$, that is, any point in $Y \setminus \Phi(x)$ has a closed neighbourhood that does not intersect $\Phi(x)$. If $\Phi^{-1}(F)$ is closed for each closed subset $F \subset Y$, then $\Phi$ is called \textit{lower-semicontinuous}. \textit{Multimorphism} of semigroups $X$ and $Y$ is a multimap $\Phi: X \multimap Y$, such that $\Phi(x)\Phi(y) \subset \Phi(xy)$ for each $x, y \in X$.

Note that in semitopological semilattice $X$ the sets ${\downarrow}x$ are continuous images of $X$. Combining Theorem \ref{locConvTheorem}, Lemma \ref{GrothPairLemma} and Theorem \ref{extTheorem} we get the following corollary.

\begin{corollary}
    {\it Let $X$ be a semitopological locally convex semilattice and $(X, X)$ be a Grothendieck pair. Then the multiplication on $X$ is continuous at each point of the diagonal $X \times X$.}
\end{corollary}

Note that although the formulation of Theorem 2.1 in \cite{BaBaChainFinite} imposes the condition of continuity of the multiplication, it is clear from the proof that its continuity only at the points of the diagonal of $X \times X$ is sufficient.

\begin{corollary}{\it
Let $X$ be a complete topologized semilattice, $Y$ be a locally convex semitopological semilattice, ${\downarrow}x$ pseudocompact and the multiplication restricted to ${\downarrow}x$ is quasicontinuous for each $x \in X$, $\Phi : X \multimap Y$ an upper-semicontinuous $T_2$-multimorphism. Then $\Phi(X)$ is a $\theta$-closed subset of $Y$.}
\end{corollary}

\bibliographystyle{model1a-num-names}
\bibliography{<your-bib-database>}

\end{document}